\documentclass{amsart}

\usepackage[english]{babel}

\renewcommand{\leq}{\leqslant}
\renewcommand{\geq}{\geqslant}
\numberwithin{equation}{section}

\usepackage[utf8]{inputenc}
\usepackage{amsmath}
\usepackage{graphicx}
\usepackage{amssymb}
\usepackage{amsthm}
\usepackage{tikz-cd}
\usepackage{mathrsfs}
\usepackage{enumitem,verbatim}

\title{Diameter bounds for degenerating Calabi-Yau metrics}

\author{Yang Li}
\address{Massachusetts Institute of Technology, 77 Massachusetts Avenue, Cambridge, MA 02139}
\email{yangmit@mit.edu}
\author{Valentino Tosatti}
\address{Department of Mathematics and Statistics, McGill University, Montr\'eal, Qu\'ebec H3A 0B9, Canada}
\email{valentino.tosatti@mcgill.ca}

\date{\today}
\newtheorem{thm}{Theorem}[section]

\newtheorem{prop}[thm]{Proposition}

\theoremstyle{definition}

\newcommand{\ddbar}{i\partial\bar{\partial}}
\newcommand{\ov}[1]{\overline{#1}}
\newcommand{\ti}[1]{\tilde{#1}}
\newcommand{\ve}{\varepsilon}

\newcommand{\Sk}{\mathrm{Sk}}

\DeclareMathOperator{\Tr}{tr}

\def\XXint#1#2#3{{\setbox0=\hbox{$#1{#2#3}{\int}$ }
		\vcenter{\hbox{$#2#3$ }}\kern-.6\wd0}}

\begin{document}
\begin{abstract}We obtain sharp upper and lower bounds for the diameter of Ricci-flat K\"ahler metrics on polarized Calabi-Yau degeneration families, as conjectured by Kontsevich-Soibelman.
\end{abstract}	
	\maketitle

\section{Introduction}
The main objects of study in this note are Ricci-flat K\"ahler metrics on Calabi-Yau manifolds whose complex structure degenerates. More precisely, we assume that we have $\pi:X\to\Delta^*\subset\mathbb{C}$ a projective holomorphic submersion with connected fibers of relative dimension $n$ with $K_{X/\Delta^*}\cong\mathcal{O}_X$ and which is meromorphic at $0$ (in the sense of \cite{BJ}), meaning that it extends to a proper flat map $\pi:\mathfrak{X}\to\Delta$ with $\mathfrak{X}$ normal. We also fix a relative polarization $L\to X$, and we will refer to this data as a polarized Calabi-Yau degeneration family, often without mentioning $L$ explicitly. The fibers $X_t$ for $t\in\Delta^*$ are thus polarized Calabi-Yau $n$-folds.

A choice of $\mathfrak{X}$ as above will be called a model of $X$. Models are highly non-unique, and in particular up to passing to a finite base change, we may assume by \cite{KKMS} that $X$ admits a semistable model, where $\mathfrak{X}$ is smooth and $X_0=\sum_{i\in I}E_i$ is reduced and has simple normal crossings. One can then apply a relative MMP to a semistable model and obtain a relatively minimal dlt model \cite{Bi,HX,KNX,NX}, which is unique up to applying sequences of flops on the central fiber \cite{Bo,Ka}. Taking the dual intersection complex of the central fiber of any such minimal dlt model, one obtains a simplicial complex $\Sk(X)$, the essential skeleton of $X$, which was introduced in this context by Kontsevich-Soibelman \cite{KS} with a different but equivalent definition (cf. \cite{MN}), whose homeomorphism type is well-defined independent of any choice of models \cite{NX}.

In this note we will not make direct use of the skeleton $\Sk(X)$ itself, but only of its real dimension which will be denoted by $m$, and which appears naturally \cite{BJ,KS} as the power of logarithmic blowup of the fiberwise integrals of the Calabi-Yau volume forms, as we will recall in Section \ref{sectprel} below. As shown in \cite{NX} we always have $m\leq n$, and the case when $m=0$ happens if and only if (after possibly a finite base change) $X$ admits a semistable model with central fiber $X_0$ which is a Calabi-Yau variety with klt singularities. Furthermore, the case $m=n$ is equivalent to the monodromy transformation around $0$ acting on $H^n(X_t,\mathbb{C})$ having a Jordan block of size $n+1$. This is the familiar notion of a ``large complex structure limit'' from mirror symmetry, see e.g. \cite{Gr}, where these polarized Calabi-Yau degeneration families play a crucial role.

Our main interest is in the behavior as $t\to 0$ of the Ricci-flat K\"ahler metrics $\omega_t$ on $X_t$ in the scaled class $\frac{1}{|\log|t||}c_1(L)|_{X_t}$, whose existence is guaranteed by Yau's Theorem \cite{Ya}. In \cite[Conjecture 1]{KS} Kontsevich-Soibelman conjectured that if $X\to\Delta^*$ is a large complex structure limit of Calabi-Yau manifolds, then the diameter of $(X_t,\omega_t)$ is bounded away from zero and infinity (note that there is a typo in the statement of their conjecture), and furthermore they, and independently also Gross-Wilson \cite{GW} and Todorov, conjectured that the collapsed Gromov-Hausdorff limit of $(X_t,\omega_t)$ is a half-dimensional affine manifold with singularities in codimension $2$, which is homeomorphic to $\Sk(X)$, and which is expected to be the base of the Strominger-Yau-Zaslow fibration of $X_t$ \cite{SYZ}, see e.g. \cite[\S 7]{Br} and \cite{To2} for surveys of these and related topics.

The main theorem of this note is to prove the conjectured sharp diameter bound in \cite[Conjecture 1]{KS}, for all polarized Calabi-Yau degeneration families with $m>0$, thus also settling \cite[Conjecture 4.7]{To2}:
\begin{thm}\label{th1}
Let $\pi:X\to\Delta^*$ be a polarized Calabi-Yau degeneration family, suppose that the dimension $m$ of the essential skeleton $\Sk(X)$ is positive, and let $\omega_t$ be the Ricci-flat K\"ahler metric on $X_t$ in the class $\frac{1}{|\log|t||}c_1(L)|_{X_t}$, for $t\in\Delta^*.$ Then there is $C>0$ such that
$$C^{-1}\leq \mathrm{diam}(X_t,\omega_t)\leq C,$$
for all $t\in\Delta^*$ with $|t|$ sufficiently small.
\end{thm}

The assumption that $m>0$ is necessary, since when $m=0$ we can find a semistable model $\mathfrak{X}$ with central fiber $X_0$ a Calabi-Yau variety with klt singularities, as mentioned above, and then it is known by work of Rong-Zhang \cite{RZ} that we have instead $\mathrm{diam}(X_t,\omega_t)\sim |\log|t||^{-\frac{1}{2}}$.

Despite several recent works addressing diameter bounds for K\"ahler-Einstein metrics under various assumptions, see e.g. \cite{FGS,Li,So}, the only previously known general results in the direction of our main theorem are the following. First, tracing through the arguments given in \cite[Theorem 2.1]{RZ} (see also \cite[Proposition 4.2]{Ta}) gives the upper bound
$$\mathrm{diam}(X_t,|\log|t||\omega_t)\leq C|\log|t||^{m},$$
which is worse than the one provided by Theorem \ref{th1}. And secondly, it follows from the earlier works \cite{Wa,To,Ta} that when $m>0$ we necessarily have $\mathrm{diam}(X_t,|\log|t||\omega_t)\to\infty$ (see also the exposition in \cite{Zh}), but the arguments there do not provide any explicit lower bound.

The rough idea of our proof is the following: as we will recall in Section \ref{sectprel}, a well-known computation in polar coordinates (cf. \cite{BJ}) reveals that most of the mass of the Calabi-Yau volume forms $\omega_t^n$ on $X_t$ is carried by ``very small'' regions which are near certain intersections of $m+1$ irreducible components of the central fiber. In Section \ref{sectlow} we then construct a K\"ahler metric $\omega'_t$ cohomologous to $\omega_t$ which behaves like a toric metric in this good region (in the directions $z_1,\dots,z_m$ normal to these components). Using $\omega'_t$ we then obtain a uniform $L^1$ bound for $|d\rho_t|_{\omega_t}$ where $\rho_t$ looks like a paraboloid in the logarithmic coordinates $x_1,\dots,x_m$ in our good region, and using Cheeger-Colding's segment inequality \cite{CC} we deduce the diameter lower bound. Lastly, in Section \ref{sectup} we again use $\omega'_t$ to produce a unit-size $\omega_t$-geodesic ball in $X_t$ whose volume is a definite fraction of the total, from which the diameter upper bound follows from an argument of Yau, as in \cite{To0,RZ}.

\medskip

\noindent
{\bf Acknowledgments. }
The first-named author is a 2020 Clay Research Fellow, and was based at the Institute for Advanced Study, supported by the Zurich Insurance Company Membership, when this article was written. He thanks Song Sun for earlier discussions. The second-named author would like to thank S.Takayama and Y.Zhang for earlier discussions on these topics, and N. McCleerey and T.-D. T\^{o} for email exchanges.  He was partially supported by NSF grant DMS-1903147, and this article was written during his visit at the Department of Mathematics and at the Center for Mathematical Sciences and Applications at Harvard University, which he would like to thank for the hospitality. We are also grateful to the referee for useful suggestions.

\section{Volume form asymptotics}\label{sectprel}
In this section we recall some background on the asymptotic behavior of the relative Calabi-Yau volume forms on a polarized Calabi-Yau degeneration family, and set up some notation for later use.

As in the introduction, we assume we have a polarized Calabi-Yau family $\pi:X\to\Delta^*$ with relative polarization $L$, and we fix a trivializing section $\Omega$ of $K_X$ and define trivializations $\Omega_t$ of $K_{X_t}$ by $\Omega=dt\wedge\Omega_t$ along $X_t$. Up to passing to a finite base change, we may assume that $X$ admits a semistable model, where $\mathfrak{X}$ is smooth and $X_0=\sum_{i\in I}E_i$ is reduced and has simple normal crossings. In this case we have
$$K_{\mathfrak{X}/\Delta}=\sum_{i\in I} a_i E_i,\quad a_i\in\mathbb{Z},$$
and letting $\kappa=\min_{i\in I} a_i$, up to replacing $\Omega_t$ by $t^{-\kappa}\Omega_t$, we may assume without loss that $\kappa=0$.

Recall that $\omega_t$ denotes the Calabi-Yau metric on $X_t$ in the class $\frac{1}{|\log |t||}c_1(L)|_{X_t}$, and that the dimension $m$ of the essential skeleton of $X$ is assumed to be stricty positive (and necessarily $m\leq n$). Denote by $\mu_t$ the Calabi-Yau volume form on $X_t$ normalized to be a probability measure, i.e.
\begin{equation}\label{volform}
\mu_t=\frac{\omega_t^n}{\int_{X_t}\omega_t^n}=\frac{i^{n^2}\Omega_t\wedge\ov{\Omega_t}}{\int_{X_t}i^{n^2}\Omega_t\wedge\ov{\Omega_t}}.
\end{equation}

Let us now recall the asymptotic behavior of the integrals $\int_{X_t}i^{n^2}\Omega_t\wedge\ov{\Omega_t}$, largely following \cite{BJ}.
For any $J\subset I$ we denote by $E_J=\bigcap_{j\in J}E_j$. As in \cite{Li2,Li3}, we fix a K\"ahler metric on $\mathfrak{X}$ and for $\ve>0$ small and $J\subset I$ with $E_J\neq\emptyset$ we define
$$E^0_J=\{x\in X_t\ |\ d(x,E_J)<\ve\}\backslash\{x\in X_t\ |\ d(x,E_{J'})<\ve\text{ for some }J'\supsetneq J\}.$$
For any given $x\in E_J\subset\mathfrak{X}$ let $p=|J|-1$ and pick local coordinates $z_0,\dots,z_n$ on $x\in V\subset\mathfrak{X}$, defined in the unit polydisc, such that $z_0,\dots,z_p$ are defining equations for $E_j, j\in J$, so that in these coordinates we have $t=z_0\cdots z_p$. We shall call these adapted coordinates. We can then write
$$\Omega=f_J \prod_{i=0}^pz_i^{a_i}dz_i\wedge\prod_{j=p+1}^n dz_j,$$
where $f_J$ is a local non-vanishing holomorphic function. Since $dt\wedge\Omega_t=\Omega$ along $X_t$, on $E^0_J$ we get
$$\Omega_t=f_Jz_0^{a_0}\cdots z_p^{a_p}\prod_{j=1}^p\frac{dz_j}{z_j}\wedge\prod_{k=p+1}^n dz_k,$$
$$i^{n^2}\Omega_t\wedge\ov{\Omega_t}=|f_J|^2|z_0|^{2a_0}\cdots |z_p|^{2a_p}\prod_{j=1}^pi\frac{dz_j}{z_j}\wedge \frac{d\ov{z_j}}{\ov{z_j}}\wedge\prod_{k=p+1}^n idz_k\wedge d\ov{z_k},$$
from which using polar coordinates $z_j=\exp(x_j\log|t|+i\theta_j),j\in J,$ one can easily see as in \cite{BJ} that
$$\int_{E^0_J} i^{n^2}\Omega_t\wedge\ov{\Omega_t}\sim |\log|t||^{m_J},$$
where
$$m_J=|\{j\in J\ |\ a_j=0\}|-1,$$
while
$$\int_{X_t} i^{n^2}\Omega_t\wedge\ov{\Omega_t}\sim |\log|t||^{m},$$
where
\begin{equation}\label{defn}
m=\max\{|J|-1\ |\ E_J\neq\emptyset, a_j=0\text{ for all }j\in J\}=\dim \Sk(X).
\end{equation}
The local logarithmic variables $x_j=\frac{\log|z_j|}{\log|t|}$ vary in the standard simplex
$$\Delta_J=\left\{0\leq x_j\leq 1\ |\ \sum_{j=0}^px_j=1\right\},$$
and in this way one obtains a map $\mathrm{Log}_t:V\backslash \bigcup_j E_j\to \Delta_J$, see \cite{BJ}.

For each $i\in I$ we fix now a defining section $\sigma_i\in H^0(\mathfrak{X},\mathcal{O}(E_i))$ and a Hermitian metric $h_i$ on $\mathcal{O}(E_i)$, so that $r_i:=|\sigma_i|^2_{h_i}$ is a smooth nonnegative function of $\mathfrak{X}$ which vanishes precisely along $E_i$ and is uniformly comparable to $|z_i|^2$ in any adapted coordinate chart as above where $z_i$ is the local defining equation of $E_i$. In particular,
$$\ti{x}_i:=\frac{\log r_i}{\log|t|},$$
is now defined on the whole $\mathfrak{X}\backslash \bigcup_j E_j$, and in the adapted coordinates as above it is equal to $2x_i$ up to very small errors (as $t$ approaches $0$). It follows that on $X_t$ (for $|t|$ sufficiently small) in an adapted coordinate chart near a point of $E^0_J$ as above, the point $\left(\frac{1}{2}\ti{x}_0,\dots,\frac{1}{2}\ti{x}_p\right)$ will lie very close to $\Delta_J$.

\section{Diameter lower bound}\label{sectlow}
In this section we prove the diameter lower bound in Theorem \ref{th1}.

The setting is the same as in the previous section, so $X\to\Delta^*$ is a polarized Calabi-Yau degeneration family with $m=\dim\Sk(X)>0$, with a semistable model $\mathfrak{X}\to\Delta$ with $X_0=\sum_{i\in I}E_i$ and $K_{\mathfrak{X}/\Delta}=\sum_{i\in I} a_iE_i$.  We fix also an embedding of the family $\mathfrak{X}\hookrightarrow \mathbb{P}^{N_0}\times\Delta$ and denote by $\mathfrak{L}\to\mathfrak{X}$ the restriction of the hyperplane bundle.

We choose a nonempty $E_J$ which realizes the maximum in \eqref{defn}, with $m=|J|-1$, and relabel so that $J=\{0,\dots,m\}$. We also denote by $U$ an open neighborhood of $E_J$ in $\mathfrak{X}$ which can be covered by finitely many adapted coordinate charts as above. In particular, in these charts we have
\begin{equation}\label{calcul}
i^{n^2}\Omega_t\wedge\ov{\Omega_t}=|f_J|^2\prod_{j=1}^mi\frac{dz_j}{z_j}\wedge \frac{d\ov{z_j}}{\ov{z_j}}\wedge\prod_{k=m+1}^n idz_k\wedge d\ov{z_k}.
\end{equation}

We need the following construction:

\begin{prop}\label{pr1}
We can find a metric $\omega_t'$ on $X_t$ in the class $\frac{1}{|\log |t||}c_1(\mathfrak{L})|_{X_t}$, a Lipschitz function $\rho_t$ on $X_t$, and an open neighborhood $U$ of $E_J$ in $\mathfrak{X}$ as above, with the following properties:
\begin{itemize}
\item[(a)]
The function $\rho_t$ is supported on the closure of $B_t=\{ \rho_t<0 \}\subset U\cap X_t$. On $B_t$ the function $\rho_t$ is comparable to a quadratic function in the logarithmic variables $x_j=\frac{\log |z_j|}{|\log |t|| }$, for $j=1,2,\ldots m$, in adapted coordinate charts, with $\min \rho_t=-1$  and $\max \rho_t= 0$.
\item[(b)]
On $B_t$ in adapted coordinate charts we have
\begin{equation}\label{ref}
\omega_t'\geq C^{-1}  \frac{i}{|\log |t||^2}\sum_{j=1}^m \frac{dz_j}{z_j} \wedge \frac{d\ov{z_j}}{\ov{z_j}},
\end{equation}
and
\begin{equation}\label{grad}
|d\rho_t|^2_{\omega'_t}\leq C,
\end{equation}
for a fixed constant $C$ independent of $t$.
\end{itemize}
\end{prop}

For ease of notation, in the rest of the paper we will denote by $C>0$ a generic uniform constant, independent of $t$, which may vary from line to line.

\begin{proof}
Given any point $x\in E_J$ we can find $k\gg 1$ and sections $s_0,\dots,s_N\in H^0(\mathfrak{X},\mathfrak{L}^k)$ (for some $N=N(k)$) so that in some adapted coordinate chart $V_x$ near $x$ we have that none of the sections $s_0, s_{m+1},\dots,s_N$ vanishes, while $s_j=0$ is a defining equation for $E_j$, $1\leq j\leq m$, and so $s_j/s_0$ is comparable to $z_j$ for $1\leq j\leq m$.

We construct a K\"ahler metric $\omega_{x,t}'$ on $X_t$ by pulling back a suitable toric metric on $\mathbb{P}^{N_0}$ , which on the complement of the zeros of all the $s_i$'s is given by
$$\omega'_{x,t}=\frac{1}{k}\ddbar u\left(\frac{\log|s_1/s_0|}{|\log|t||},\dots,\frac{\log|s_N/s_0|}{|\log|t||}\right),$$
where $u(x_1,\dots,x_N)$ is a smooth convex function in $\mathbb{R}^N$ which is asymptotic to $v(x_1,\dots,x_N)=\max(0,x_1,\dots,x_N)$ at infinity, and with $D^2 u \geq C^{-1}\mathrm{Id}$ on a ball of radius comparable to 1 containing the image of $V_x\cap X_t$ in the logarithmic coordinates. For example, an explicit such $u$ can be produced as the convolution of $v$ with a smooth mollifier. By construction, $\omega'_{x,t}$ lies in the class $\frac{1}{|\log |t||}c_1(\mathfrak{L})|_{X_t}$, and it satisfies \eqref{ref} on $V_x\cap X_t$.

We then choose finitely many $x^{(1)},\dots,x^{(M)}\in E_J$ such that the corresponding $V_{x^{(1)}},\dots,V_{x^{(M)}}$ cover $E_J$, let $U$ be their union, and define
$$\omega'_t=\frac{1}{M}\sum_{k=1}^M\omega'_{x^{(k)},t}.$$
This is our desired K\"ahler metric on $X_t$ in $\frac{1}{|\log |t||}c_1(\mathfrak{L})|_{X_t}$ which satisfies \eqref{ref} in adapted coordinate charts on $U\cap X_t$.

Next, we consider the function
$$\hat{\rho}=A\sum_{j=1}^m \left(\frac{\ti{x}_j}{2}-b_j\right)^2-1=A\sum_{j=1}^m\left(\frac{\log r_j}{2\log|t|}-b_j\right)^2-1,$$
on $U\backslash \bigcup_{i\in I} E_i$.
Choosing the constants $b_j$ in the strict interior of $\Delta_J$ we can ensure the minimum of $\hat{\rho}$ on $U\cap X_t$ equals $-1$, and choosing
 $A$ suitably large independent of $t$, we can ensure $\{ \hat{\rho}\leq 0  \}$ is compactly contained in $U$. Now we define
$$\rho_t=\begin{cases}
\min(\hat{\rho},0)|_{X_t}\quad &\text{ on }U\cap X_t\\
0\quad &\text{ on }X_t\backslash U
\end{cases},$$
which satisfies the requirements in (a). We let $B_t=\{\rho_t<0\}\subset X_t$. Lastly, \eqref{grad} follows immediately from part (a) and \eqref{ref}.
\end{proof}

We can now give the proof of the diameter lower bound in Theorem \ref{th1}:
\begin{proof}[Proof of the diameter lower bound in Theorem \ref{th1}]
Thanks to Proposition \ref{pr1}, on $B_t\subset X_t$ we have
$$|d\rho_t|_{\omega_t'}^2\leq C,$$
for some constant $C$ independent of $t$. We then use this together with the elementary inequality $|d\rho_t|_{\omega_t}^2\leq |d\rho_t|_{\omega_t'}^2 \Tr_{\omega_t}\omega_t'$ to get
\[
\left( \int_{X_t} |d\rho_t|_{\omega_t} d\mu_t\right)^2=\left(\int_{B_t} |d\rho_t|_{\omega_t} d\mu_t\right)^2\leq
\int_{B_t} |d\rho_t|_{\omega_t}^2 d\mu_t\leq C\int_{X_t} \Tr_{\omega_t} \omega_t'd\mu_t,
\]
while from \eqref{volform} we get
$$\int_{X_t} \Tr_{\omega_t} \omega_t'd\mu_t=\frac{n\int_{X_t}\omega'_t\wedge\omega_t^{n-1}}{\int_{X_t}\omega_t^n}=\frac{n\int_{X_t}c_1(\mathfrak{L})\cdot
c_1(L)^{n-1}}{\int_{X_t}c_1(L)^n}\leq C,$$
and so
\begin{equation}\label{part1}
 \int_{X_t} |d\rho_t|_{\omega_t} d\mu_t\leq C.
\end{equation}
Define two subsets of $X_t$ by $A_1= \{ \rho_t< -\frac{2}{3}  \}$ and $A_2=\{  -\frac{1}{3}\leq \rho_t\leq 0  \}$. Given two points $x\in A_1, y\in A_2$ which are connected by a unique minimal geodesic $\gamma_{x,y}$ (w.r.t. $\omega_t$), we can bound
\begin{equation}\label{part2}
\rho_t(y)-\rho_t(x)\leq \int_{\gamma_{x,y}} |d\rho_t|_{\omega_t} ds,
\end{equation}
where $\gamma_{x,y}$ is parametrized with respect to $\omega_t$-arclength.

Combining \eqref{part2} with Cheeger-Colding's segment inequality \cite[Theorem 2.11]{CC} applied to the function $|d\rho_t|_{\omega_t}$ we obtain
\begin{equation}\label{part3}
\begin{split}
D_t( \mu_t(A_1)+ \mu_t(A_2) ) \int_{X_t} |d\rho_t|_{\omega_t} d\mu_t&\geq C^{-1}\int_{A_1\times A_2} \left(\int_{\gamma_{x,y}} |d\rho_t|_{\omega_t}ds\right)d\mu_xd\mu_y\\
&\geq C^{-1}\int_{A_1\times A_2} (\rho_t(y)-\rho_t(x)) d\mu_xd\mu_y\\
&\geq \frac{C^{-1}}{3}\mu_t(A_1)\mu_t(A_2),
\end{split}
\end{equation}
where $D_t=\mathrm{diam}(X_t,\omega_t)$, and in the $\int_{A_1\times A_2}$ we are actually only integrating over the subset of pairs $(x,y)$ which are joined by a unique $\omega_t$-minimal geodesic, which has full measure (cf. \cite{CC}).

Combining \eqref{part1} and \eqref{part3} gives
$$\mu_t(A_1)\mu_t(A_2)\leq C D_t( \mu_t(A_1)+ \mu_t(A_2) )\leq C D_t.$$
Lastly, from the definition of $\rho_t$ and from \eqref{calcul}, a direct computation in polar coordinates (analogous to the one in \cite{BJ}) gives
\[
\mu_t(A_1) \geq C^{-1},  \quad \mu_t(A_2) \geq C^{-1},
\]
for a fixed constant $C$, and so $D_t\geq C^{-1},$ as desired.

\end{proof}

\section{Diameter upper bound}\label{sectup}

Here we prove the diameter upper bound in Theorem \ref{th1}.

\begin{proof}[Proof of the diameter upper bound in Theorem \ref{th1}]
Let $\omega'_t$ be the K\"ahler metric defined in Proposition \ref{pr1}, and let $\omega_{{\rm FS},t}=\frac{1}{|\log|t||}\omega_{\rm FS}|_{X_t},$ where $\omega_{\rm FS}$ is a suitable Fubini-Study metric on $\mathbb{P}^{N_0}$ scaled so that $\omega_{\rm FS}|_{X_t}\in c_1(\mathfrak{L})|_{X_t}$. Then the metrics $\omega'_t$ and $\omega_{{\rm FS},t}$ are cohomologous, and on $B_t\subset X_t$ in any adapted coordinate chart we have
\begin{equation}\label{lb}
\omega'_t+\omega_{{\rm FS},t}\geq  C^{-1}  \frac{i}{|\log |t||^2}\sum_{j=1}^m \frac{dz_j}{z_j}\wedge \frac{d\ov{z_j}}{\ov{z_j}}
+C^{-1}  \frac{i}{|\log |t||}\sum_{j=m+1}^n dz_j \wedge d\overline{z_j}.
\end{equation}
Let us then fix a point $x\in E^0_J$, an adapted coordinate chart near $x$, and in these coordinates define a local K\"ahler metric $\ti{\omega}_t$ on $X_t$ by the RHS of \eqref{lb}. Inside this coordinate chart intersected $X_t$, define also $\ti{B}_t$ to be a Euclidean rectangle which is contained inside $B_t$ so that in the metric $\ti{\omega}_t$, in the first $m$ complex directions $\ti{B}_t$ has length $\sim 1$ in the radial directions and length $\sim|\log|t||^{-1}$ in the logarithmic angular directions, and in the other $n-m$ complex directions $\ti{B}_t$ has length $\sim |\log|t||^{-\frac{1}{2}}$. Therefore, given any $x,y\in \ti{B}_t$, if we denote by $\gamma_{x,y}$ the Euclidean straight line in $\ti{B}_t$ joining them, parametrized linearly by $0\leq s\leq 1$, then
$|\dot{\gamma}_{x,y}(s)|_{\ti{\omega}_t}\leq C$ for a uniform constant $C$ independent of $t$ and $s$.

On $\ti{B}_t$ we have
$$\frac{C^{-1}}{|\log|t||^n}\mu_t\leq \ti{\omega}_t^n\leq \frac{C}{|\log|t||^n}\mu_t,$$
and again by direct computation in polar coordinates (as in \cite{BJ}), thanks to the definition of $\ti{B}_t$ and to \eqref{calcul} we have that
\begin{equation}\label{totalint}
C^{-1}\leq \int_{\ti{B}_t}d\mu_t\leq 1,
\end{equation}
and so
$$\frac{C^{-1}}{|\log|t||^n}\leq \int_{\ti{B}_t}\ti{\omega}_t^n\leq \frac{C}{|\log|t||^n}.$$
If we then define $$\ti{\mu}_t=\frac{\ti{\omega}^n_t}{\int_{\ti{B}_t}\ti{\omega}_t^n},$$
then $\ti{\mu}_t$ is uniformly comparable to $\mu_t$ on $\ti{B}_t$ and
\begin{equation}\label{totalint2}
\mu_t(\ti{B}_t)\geq C^{-1}\mu_t(B_t)\geq C^{-1}.
\end{equation}

Thanks to \eqref{lb} we have
\begin{equation}\label{ub}
\begin{split}
\int_{\ti{B}_t}\Tr_{\ti{\omega}_t} \omega_t d\ti{\mu}_t&=\frac{n\int_{\ti{B}_t}\omega_t\wedge\ti{\omega}_t^{n-1}}{\int_{\ti{B}_t}\ti{\omega}^n_t}\leq C|\log|t||^n\int_{X_t}
\omega_t\wedge(\omega'_t+\omega_{{\rm FS},t})^{n-1}\\
&\leq C\int_{X_t}c_1(L)\cdot c_1(\mathfrak{L})^{n-1}\leq C.
\end{split}\end{equation}
We wish to use this to prove that
\begin{equation}\label{l1}
\int_{\ti{B}_t\times \ti{B}_t} \text{dist}_{\omega_t} (x,y) d\ti{\mu}_t(x)d\ti{\mu}_t(y) \leq C.
\end{equation}
Indeed, since $|\dot{\gamma}_{x,y}(s)|_{\ti{\omega}_t}\leq C$, we can estimate
$$\mathrm{dist}_{\omega_t}(x,y)\leq C\int_{0}^1(\Tr_{\ti{\omega}_t}{\omega_t}(sy+(1-s)x))^\frac{1}{2}ds,$$
and
\[\int_{\ti{B}_t\times \ti{B}_t}\text{dist}_{\omega_t} (x,y) d\ti{\mu}_t(x)d\ti{\mu}_t(y)\leq C\int_{\ti{B}_t\times \ti{B}_t}\int_{0}^1(\Tr_{\ti{\omega}_t}{\omega_t}(sy+(1-s)x))^\frac{1}{2}ds\, d\ti{\mu}_t(x)d\ti{\mu}_t(y).\]
We can then argue as in \cite[Lemma 1.3]{DPS}, using Fubini
\[\begin{split}
&\int_{\ti{B}_t\times \ti{B}_t}\int_{0}^1(\Tr_{\ti{\omega}_t}{\omega_t}(sy+(1-s)x))^\frac{1}{2}ds\, d\ti{\mu}_t(x)d\ti{\mu}_t(y)\\
&=\int_{0}^1\int_{\ti{B}_t\times \ti{B}_t}(\Tr_{\ti{\omega}_t}{\omega_t}(sy+(1-s)x))^\frac{1}{2}d\ti{\mu}_t(x)d\ti{\mu}_t(y)ds\\
&=\int_0^{\frac{1}{2}}\int_{\ti{B}_t\times \ti{B}_t}(\Tr_{\ti{\omega}_t}{\omega_t}(sy+(1-s)x))^\frac{1}{2}d\ti{\mu}_t(x)d\ti{\mu}_t(y)ds\\
&+\int_{\frac{1}{2}}^1\int_{\ti{B}_t\times \ti{B}_t}(\Tr_{\ti{\omega}_t}{\omega_t}(sy+(1-s)x))^\frac{1}{2}d\ti{\mu}_t(y)d\ti{\mu}_t(x)ds
\end{split}\]
and in the two innermost integrals we change variable from $x$ (resp. $y$) to $z=sy+(1-s)x$ with $0\leq s\leq \frac{1}{2}$ (resp. $\frac{1}{2}\leq s\leq 1$), noting that $\ti{\mu}_t(x)\leq C\ti{\mu}_t(z)$ (resp. $\ti{\mu}_t(y)\leq C\ti{\mu}_t(z)$). Thus both of these innermost integrals can be bounded above by
\[C\int_{\ti{B}_t}(\Tr_{\ti{\omega}_t}{\omega_t}(z))^\frac{1}{2}d\ti{\mu}_t(z)\leq C\left(\int_{\ti{B}_t}\Tr_{\ti{\omega}_t}{\omega_t}(z)d\ti{\mu}_t(z)\right)^{\frac{1}{2}}\leq C,\]
by \eqref{ub}, and \eqref{l1} follows.

But \eqref{l1} is equivalent to
\[
\int_{\ti{B}_t\times \ti{B}_t} \text{dist}_{\omega_t} (x,x') d\mu_t(x)d\mu_t(x') \leq C,
\]
hence for some $x'\in \ti{B}_t$ we have
\[
\int_{\ti{B}_t} \text{dist}_{\omega_t} (x,x') d\mu_t(x) \leq C,
\]
which gives a weak $L^1$-estimate
\[\mu_t\{x\in \ti{B}_t\ |\ \text{dist}_{\omega_t} (x,x')\geq r\}\leq\frac{C}{r}.\]
Taking $r\gg 1$ independent of $t$,  we can ensure that
\[\begin{split}
\mu_t(B_{\omega_t}(x',r))&\geq \mu_t(\ti{B}_t\cap B_{\omega_t}(x',r))\\
&=\mu_t(\ti{B}_t)-\mu_t\{x\in \ti{B}_t\ |\ \text{dist}_{\omega_t} (x,x')\geq r\}\\
&\geq C^{-1}-\frac{C}{r}\\
&\geq C^{-1}=C^{-1}\mu_t(X_t),
\end{split}\]
using here \eqref{totalint2}.
The Bishop-Gromov volume comparison theorem then gives us that $\mu_t(B_{\omega_t}(x',1))\geq C^{-1}\mu_t(X_t),$ or equivalently
\[\frac{\int_{X_t}\omega_t^n}{\int_{B_{\omega_t}(x',1)}\omega_t^n}\leq C.\]
This bound can then be inserted into a well-known result of Yau (see e.g. \cite[Lemma 3.2]{To0}): given a complete $d$-dimensional Riemannian manifold $(M,g)$ with nonnegative Ricci curvature, for any $x\in M$ and $1<R<\mathrm{diam}(M,g)$ we have
\[\frac{R-1}{2d}\leq \frac{\mathrm{Vol}(B_g(x,2R+2))}{\mathrm{Vol}(B_g(x,1))}.\]
Indeed, it suffices to choose $R=\mathrm{diam}(X_t,\omega_t)-1$ to obtain the desired uniform diameter upper bound for $(X_t,\omega_t)$.
\end{proof}

\end{document}